\documentclass[11pt,a4paper,openright]{article}
\usepackage[english]{babel}
\usepackage[utf8]{inputenc}
\usepackage{courier}
\usepackage{amsmath,amssymb,amsthm}
\usepackage{bm}
\usepackage{paralist}
\usepackage{enumitem}
\usepackage{graphics}
\usepackage{graphicx}
\usepackage{tikz}
\usepackage{pst-plot,pst-eucl}
\usepackage[makeroom]{cancel}
\usepackage{algorithm}
\usepackage[noend]{algpseudocode}
\usepackage{faktor}
\usepackage{tikz-cd}
\usetikzlibrary{shapes.geometric}
\usepackage[top=4cm,bottom=3.5cm,left=3.3cm,right=3.3cm,asymmetric]{geometry}

\usepackage{url}
\usepackage{lipsum}
\usepackage{emptypage}
\usepackage{setspace}
\usepackage{multirow}
\usepackage[all]{xy}
\usepackage{url}
\usepackage{hyperref}
\usetikzlibrary{arrows.meta}
\makeatletter
\def\BState{\State\hskip-\ALG@thistlm}
\makeatother
\newtheoremstyle{dotless}{}{}{\itshape}{}{\bfseries}{}{ }{}
\theoremstyle{dotless}
\usepackage{stackengine}
\stackMath

\newtheorem{teor}{Theorem}[section]
\newtheorem{cor}[teor]{Corollary}
\newtheorem{lemma}[teor]{Lemma}

\newtheorem{Charac}[teor]{Characterization}
\newtheorem{defn}[teor]{Definition}
\newtheorem{cProbl}{Problem}
\newtheorem{oss}[teor]{Remark}
\newtheorem{ese}[teor]{Example}

\newcommand\scalemath[2]{\scalebox{#1}{\mbox{\ensuremath{\displaystyle #2}}}}

\def\N {{\mathbb N}}

\def\Z {{\mathbb Z}}

\def\b{\beta}

\def\sS{\mathcal{S}}

\def\shad{\textup{Shad}}

\def\supp{\textup{supp}}

\def\indeg{\textup{indeg}}
\def\corn{\textup{Corn}}

\def\slex{\textup{slex}}

\let\emptyset\varnothing

\hypersetup{
	colorlinks   = true,
	citecolor    = blue,
	filecolor=blue,
	citecolor = blue,      
	urlcolor=blue,
	allcolors=blue
}

\newcommand{\boxalign}[2][0.4\textwidth]{
	\par\noindent\tikzstyle{mybox} = [draw=black,inner sep=7pt]
	\begin{center}\begin{tikzpicture}
		\node [mybox] (box){%
			\begin{minipage}{#1}{\vspace{-6.5mm}#2}\end{minipage}
		};
		\end{tikzpicture}\end{center}
}

\newcommand\blfootnote[1]{%
  \begingroup
  \renewcommand\thefootnote{}\footnote{#1}%
  \addtocounter{footnote}{-1}%
  \endgroup
}

\begin{document}
\title{\bf Upper bounds for Extremal Betti Numbers of $t$-Spread Strongly Stable Ideals}	
\author{Luca Amata, Antonino Ficarra, Marilena Crupi\\
	\ \\
	{\footnotesize Department of Mathematical and Computer Sciences, Physical and Earth Sciences}\\
	{\footnotesize University of Messina}\\
	{\footnotesize Viale Ferdinando Stagno d'Alcontres 31, 98166 Messina, Italy}\\
	{\footnotesize e-mail: lamata@unime.it, antficarra@unime.it, mcrupi@unime.it}
}
\maketitle

\begin{abstract}

We study the extremal Betti numbers of the class of $t$--spread strongly stable ideals. More precisely, we determine the maximal number of admissible extremal Betti numbers for such ideals, and thereby we generalize the known results for $t\in \{1,2\}$. 
%	Let $K$ be a field and let $S=K[x_1,\dots,x_n]$ be the polynomial ring in $n$ variables over $K$. Extremal Betti numbers of special squarefree monomial ideals of $S$, known as the $t$-spread strongly stable ideals ($t$ is a positive integer) are studied. More precisely, we determine the maximal number of admissible extremal Betti numbers for such ideals and thereby we complete and generalize the known results for $t\in \{1,2\}$.

\blfootnote{
\hspace{-0,3cm} \emph{Keywords:} monomial ideals, minimal graded resolution, extremal Betti numbers, $t$-spread ideals.


}
\end{abstract}

\section{Introduction}
Let $S=K[x_1,\dots,x_n]$ be the standard polynomial ring in $n$ variables with coefficients in a field $K$. 
A squarefree monomial ideal $I$ of $S$ is an ideal generated by squarefree monomials. This class of ideals plays an important role in Commutative Algebra, not only for its intrinsic value but overall for its strong connections to Combinatorics and Topology. Recently, Ene, Herzog, and Qureshi \cite{EHQ} have generalized the notion of (squarefree) monomial ideal by introducing the class of $t$--spread monomial ideals.

Let $t\ge0$ be an integer, a monomial $x_{i_1}x_{i_2}\cdots x_{i_d}$, with $1\le i_1\le i_2\le\dots\le i_d\le n$, is called $t$--spread if $i_{j+1}-i_j\ge t$ for all $j=1,\dots,d-1$. A $t$--spread monomial ideal is an ideal generated by $t$--spread monomials. In the last years, many authors have focused their attention on such a class of monomial ideals in order to analyze the main algebraic invariants that may be associated to a graded ideal in $S$ and thereby they have generalized some classical results \cite{AC7, AEL,CAC, EHQ, RD}. Indeed, a $0$--spread monomial ideal is a monomial ideal, whereas a $1$--spread ideal is a squarefree monomial ideal. 

Hence, one can guess that many results on classes of ordinary (squarefree) monomial ideals may continue to be true for  classes of $t$--spread monomial ideals.

Among all the algebraic invariants of a graded ideal $I$ of $S$, the role of the graded Betti numbers is unquestionable. An important subset of the graded Betti numbers of $I$ consists of the \emph{extremal Betti numbers} (Definition \ref{def:extr}) introduced in \cite{BCP} as a refinement of the Castelnuovo-Mumford regularity and the projective dimension of the ideal $I$. Many characterizations of these graded Betti numbers for classes of monomial ideals can be found in \cite{HSV, MC, MC3, CF4, AC, AC6, AC7} and in the references therein.

Our aim is to solve the following problem.

\begin{cProbl}\em  \label{probli} 
	Given two positive integers $t$ and $n$, let $\mathcal{S}_{t,n}$ be the set of all $t$--spread strongly stable ideals in $S=K[x_1,\dots,x_n]$. What is the maximal number of extremal Betti numbers allowed for an ideal in $\mathcal{S}_{t,n}$?
\end{cProbl}

In \cite{AC6}, the authors determined the maximal number of such graded Betti numbers for a $1$--spread strongly stable ideal, whereas in \cite{AC7}, the same authors computed the maximal number of extremal Betti numbers of a $2$--spread strongly stable ideal of initial degree $2$, where for the initial degree of a graded ideal $I$ of $S$, denoted by $\indeg I$, we mean the minimum $j$ such that $I_j\ne 0$ ($I_j$ is the $K$-vector space spanned by the homogeneous elements of $I$ of degree $j$). The problem of determining such number for $t>2$ is still open. In this paper, we are able to generalize the results in \cite{AC6, AC7} giving a positive answer to Problem \ref{probli}. We determine the maximal number of extremal Betti numbers of a $t$--spread strongly stable ideal of initial degree $\ge 2$, for all integers $t\ge 2$.

The plan of the paper is the following. Section \ref{sec2} contains some preliminary material that will be used through the paper. In Section \ref{sec4}, we construct a suitable set of $t$--spread monomials of $S$ (Subsections \ref{sub1}, \ref{sub2}) which allow us to solve Problem \ref{probli}. We provide some examples illustrating our techniques.
In Section \ref{sec5}, we prove our main result (Theorem \ref{main Teor}). It establishes what is the maximum number of extremal Betti numbers of a $t$--spread strongly stable ideal $I$ of $S$ of initial degree $2$.
As a consequence of this theorem we obtain the results stated in \cite{AC6} and \cite{AC7} for $t=1, 2$, respectively. In Section \ref{sec6}, an analogous result of Theorem \ref{main Teor} is stated (Theorem \ref{generalcase}). Such result is true when the initial degree of $t$--spread strongly stable ideal $I$ of $S$ is greater than 2. Finally, Section \ref{sec7} contains our conclusions and perspectives. Almost all the examples in the paper have been verified with specific packages of \emph{Macaulay2} \cite{GDS} some of which developed by the authors of the paper. 

\section{Preliminaries}
\label{sec2}

Let $S=K[x_1,\dots,x_n]$ be the standard polynomial ring in $n$ variables with coefficients in $K$. $S$ is an $\N$--graded ring where $\deg x_i=1$, for all $i=1,\dots,n$. A monomial ideal $I$ of $S$ is an ideal generated by monomials. By $G(I)$ we denote the unique minimal set of monomial generators of $I$. 
For a monomial $u\in S$, $u\ne 1$, we set
\[\supp(u)=\big\{i:x_i\ \text{divides}\ u\big\},\]
and write
\[\max(u)=\max\big\{i:i\in\supp(u)\big\},\,\,
\min(u)=\min\big\{i:i\in\supp(u)\big\}.\]
Moreover, we set $\max(1)=\min(1)=0$.
\begin{defn}
	\rm Let $t\ge0$ be an integer. A monomial $x_{i_1}x_{i_2}\cdots x_{i_d}$ with $1\le i_1\le i_2\le\dots\le i_d\le n$ is called $t$--spread, if $i_{j+1}-i_j\ge t$, for all $j=1,\dots,d-1$. A monomial ideal $I$ of $S$ is called a $t$--spread monomial ideal, if it is generated by $t$--spread monomials.
\end{defn}
\par\noindent
For instance, $x_1x_3x_6\in K[x_1,x_2,x_3,x_4,x_6]$ is a $2$--spread monomial, but not a $3$--spread monomial.  Every monomial is $0$--spread and every monomial ideal is a $0$--spread monomial ideal. A squarefree monomial is a $1$--spread monomial and a squarefree monomial ideal is a $1$--spread monomial ideal. If $t\ge 1$, then every $t$--spread monomial is a squarefree monomial.
\begin{defn}
	\rm A $t$--spread monomial ideal $I$ of $S$ is called $t$--spread stable, if for all $t$--spread monomials $u \in I$ and
	for all $i < \max(u)$ such that $x_i(u/x_{\max(u)})$ is a $t$--spread monomial, it follows that $x_i(u/x_{\max(u)})\in I$.
	The ideal $I$ is called $t$--spread strongly stable, if for all $t$--spread monomials $u\in I$, all $j\in\supp(u)$ and all
	$i < j$ such that $x_i(u/x_j)$ is $t$--spread, it follows that $x_i(u/x_j) \in I$.
\end{defn}

Let $u_1,\dots,u_r$ be $t$--spread monomials of $S$. The unique $t$--spread strongly stable ideal containing $u_1,\dots,u_r$ will be denoted by $B_t(u_1,\dots,u_r)$ \cite{EHQ}. The monomials $u_1,\dots,u_r$ are called $t$--spread Borel generators, and $B_t(u_1,\dots,u_r)$ is called the finitely generated $t$--spread Borel ideal.

Let us denote by $M_{n,d, t}$ the set of all $t$--spread monomials of degree $d$ in $S=K[x_1,\dots,x_n]$. From \cite[Theorem 2.3]{EHQ}, the cardinality of $M_{n,d,t}$ is given by
$$
|M_{n,d,t}|=\binom{n-(d-1)(t-1)}{d}.
$$

Let $t\ge 1$. We endow the set $M_{n,d,t}$ with the \emph{squarefree lexicographic order}, $\ge_{\slex}$ \cite{AHH2}. More precisely, let
$u=x_{i_1}x_{i_2}\cdots x_{i_d}$ and $v=x_{j_1}x_{j_2}\cdots x_{j_d}$,
be $t$--spread monomials of degree $d$, with $1\le i_1<i_2<\dots<i_d\le n$, $1\le j_1<j_2<\dots<j_d\le n$, then
$u>_{\slex}v$ if $i_1=j_1,\dots,i_{s-1}=j_{s-1}$ and $i_s<j_s$,
for some $1\le s\le d$. 

It's easy to verify that if $u$ is a $t$--spread monomial of $S$, then for all $v\in B_t(u)$ we have $v\ge_{\slex}u$.

Note that the existence of a $t$--spread monomial of degree $d$ in $S$ implies that $n\ge (d-1)t+1$. Indeed, the monomial
$x_1x_{1+t}x_{1+2t}\cdots x_{(d-1)t+1}$ is the greatest $t$--spread monomial of $M_{n,d,t}$, with respect to $>_{\slex}$.

Let $T$ be a not empty subset of $M_{n,d,t}$. We denote by $\max T$ ($\min T$, respectively) the maximal (minimum, respectively) monomial $w\in T$, with respect to $>_{\slex}$.

From now on, we assume that $M_{n,d,t}$ ($t\ge 1$) is endowed by the \emph{squarefree lexicographic order}.\\

Furthermore, we define the $t$--shadow of $T$
\begin{eqnarray*}
	\shad_t(T)&=&\Big\{x_iw\,:\,w\in T,i=1,\dots,n\Big\}\cap M_{n,d+1,t} \\
	&=& \big\{x_iw\,:\, w\in T\, \mbox{and $x_iw$ is $t$--spread monomial, $i=1,\dots,n$}\big\}.
\end{eqnarray*}

The set $\shad_t(T)$ could be empty. We define $\shad_t^1(T)=\shad_t(T)$ and $\shad_t^n(T)=\shad_t(\shad_t^{n-1}(T))$ for all $n\ge2$, by induction.
\medskip

If $I$ is a $t$--spread strongly stable ideal, then the graded Betti numbers of $I$ can be computed by  \cite[Corollary 1.12]{EHQ}
\begin{equation}
\label{eq1}
\b_{k,k+\ell}(I)=\sum_{u\in G(I)_\ell}\binom{\max(u)-t(\ell-1)-1}{k}.
\end{equation}
Such a formula returns the Eliahou--Kervaire formula \cite{EK} for (strongly) stable ideals  whenever $t=0$ and 
the Aramova--Herzog--Hibi formula \cite{AHH2, JT} for squarefree (strongly) stable ideals whenever $t=1$.

\begin{defn}\label{def:extr}
	\rm (\cite{BCP}) A graded Betti number $\b_{k,k+\ell}(I)\ne 0$ is called extremal if $\b_{i,i+j}(I)=0$ for all $i\ge k$, $j\ge\ell,(i,j)\ne(k,\ell)$.
	
	The pair $(k,\ell)$ is called a corner of $I$.
	
	If $(k_1,\ell_1),\dots,(k_r,\ell_r)$ ($n-1\ge k_1>k_2>\dots>k_r\ge1$, $1\le \ell_1<\ell_2<\dots<\ell_r$) are the corners of a graded ideal $I$ of $S$, the set
	$$
	\corn(I)=\Big\{(k_1,\ell_1),(k_2,\ell_2),\dots,(k_r,\ell_r)\Big\}
	$$
	is called the corner sequence of $I$ \cite{MC}; whereas the $r$-uple
	$$
	a(I)=\big(\b_{k_1,k_1+\ell_1}(I),\b_{k_2,k_2+\ell_2}(I),\dots,\b_{k_r,k_r+\ell_r}(I)\big)
	$$
	is called the corner values sequence of $I$ \cite{MC}.
\end{defn}
We conclude this section by quoting two results from \cite{AC7}.
\begin{Charac}
	\label{betti teor}
	\textup{(\cite[Theorem 1]{AC7})}
	Let $I$ be a $t$--spread strongly stable ideal of $S$. The following conditions are equivalent:
	\begin{enumerate}[label=\textup{(\alph*)}]
		\item $\b_{k,k+\ell}(I)$ is extremal;
		\item $k+t(\ell-1)+1=\max\{\max(u):u\in G(I)_\ell\}$ and $\max(u)<k+t(j-1)+1$, for all $j>\ell$ and for all $u\in G(I)_j$.
	\end{enumerate}
\end{Charac}
\begin{cor}
	\textup{(\cite[Corollary 2]{AC7})} Let $I$ be a $t$--spread strongly stable ideal of $S$ and let $\b_{k,k+\ell}(I)$ be an extremal Betti number of $I$. Then
	$$
	\b_{k,k+\ell}(I)=\Big|\Big\{ u\in G(I)_\ell:\max(u)=k+t(\ell-1)+1 \Big\}\Big|.
	$$
\end{cor}

\section{Corners of $t$--Spread Strongly Stable Ideals of initial degree $2$}
\label{sec4}

In this Section, if $S=K[x_1,\dots,x_n]$, we manage some suitable $t$--spread monomials of $S$ in order to examine the behavior of the corners of a $t$--spread strongly stable ideal of $S$  of initial degree $2$.

Let us denote by $\sS_{t,n}$ the set of all $t$--spread strongly stable ideals in $S$ and by $\sS_{t,n,\bm{1}}$ the set of all  $I\in\sS_{t,n}$ for which the value of every extremal Betti number equals $1$, \emph{{\it i.e.}}, all the entries of the corner values sequence $a(I)$ are equal to $1$:
$$
\sS_{t,n,\bm{1}}=\Big\{ I\in\sS_{t,n}:a(I)=\bm{1}=(1,1,\dots,1) \Big\}.
$$

Our goal is to determine the greatest admissible number of corners for an ideal of initial degree $2$ in $\sS_{t,n,\bm{1}}$.

%To gather intuition, 
The starting point of our work has been the analysis of several examples ($t=2$, $3$, $4$, $5$) using the computer algebra system \emph{Macaulay2}. In each of these cases, we have fixed two positive integers $n$ and $\ell_1$, and using techniques similar to those in \cite{AC6} and \cite{AC7}, we have determined the maximum number of admissible corners of a $t$--spread strongly stable ideal $I$ of a polynomial ring in $n$ variables and such that $\indeg I=\ell_1$. Then, all the data obtained have been collected in some tables to analyze how the maximum number of corners varied with respect to the parameters (see for instance, Table~\ref{tab:1} and Table~\ref{tab:2}).

\begin{table}[H]
	\[
	\scalemath{0.9}{\begin{array}{|c|c|c|c|c|c|c|c|c|c|c|c|c|c|c|c|c|c|c|}
		\cline{3-19}
		\multicolumn{2}{c}{\multirow{2}{*}{$t=2$}} &\multicolumn{17}{|c|}{\textit{n}}\\
		\cline{3-19}
		\multicolumn{2}{c|}{} & \mathbf{4} & \mathbf{5} & \mathbf{6} & \mathbf{7} & \mathbf{8} & \mathbf{9} & \mathbf{10} & \mathbf{11} & \mathbf{12} & \mathbf{13} & \mathbf{14} & \mathbf{15} & \mathbf{16} & \mathbf{17} & \mathbf{18} & \mathbf{19} & \mathbf{20}\\
		\cline{1-19}
		\multirow{9}{*}{$\ell_1$}
		& \mathbf{2} & 1 & 1 & 2 & 2 & 2 & 3 & 3 & 4 & 4 & 5 & 5 & 6 & 6 & 7 & 7 & 8 & 8 \\
		\cline{2-19}
		& \mathbf{3} & - & - & 1 & 1 & 2 & 2 & 3 & 3 & 4 & 4 & 5 & 5 & 6 & 6 & 7 & 7 & 8 \\
		\cline{2-19}
		& \mathbf{4} & - & - & - & - & 1 & 1 & 2 & 2 & 3 & 3 & 4 & 4 & 5 & 5 & 6 & 6 & 7 \\
		\cline{2-19}
		& \mathbf{5} & - & - & - & - & - & - & 1 & 1 & 2 & 2 & 3 & 3 & 4 & 4 & 5 & 5 & 6 \\
		\cline{2-19}
		& \mathbf{6} & - & - & - & - & - & - & - & - & 1 & 1 & 2 & 2 & 3 & 3 & 4 & 4 & 5 \\  
		\cline{2-19}
		& \mathbf{7} & - & - & - & - & - & - & - & - & - & - & 1 & 1 & 2 & 2 & 3 & 3 & 4 \\
		\cline{2-19}
		& \mathbf{8} & - & - & - & - & - & - & - & - & - & - & - & - & 1 & 1 & 2 & 2 & 3 \\ 
		\cline{2-19}
		& \mathbf{9} & - & - & - & - & - & - & - & - & - & - & - & - & - & - & 1 & 1 & 2 \\   
		\cline{2-19}
		& \mathbf{10} & - & - & - & - & - & - & - & - & - & - & - & - & - & - & - & - & 1 \\     
		\cline{1-19}
		\end{array}}
	\]
	\caption{\label{tab:1} Maximum number of corners of 2--spread strongly stable ideals}
\end{table}

\begin{table}[H]
	\[
	\scalemath{0.9}{\begin{array}{|c|c|c|c|c|c|c|c|c|c|c|c|c|c|c|c|c|c|c|}
		\cline{3-19}
		\multicolumn{2}{c}{\multirow{2}{*}{$t=3$}} &\multicolumn{17}{|c|}{\textit{n}}\\
		\cline{3-19}
		\multicolumn{2}{c|}{} & \mathbf{4} & \mathbf{5} & \mathbf{6} & \mathbf{7} & \mathbf{8} & \mathbf{9} & \mathbf{10} & \mathbf{11} & \mathbf{12} & \mathbf{13} & \mathbf{14} & \mathbf{15} & \mathbf{16} & \mathbf{17} & \mathbf{18} & \mathbf{19} & \mathbf{20}\\
		\cline{1-19}
		\multirow{6}{*}{$\ell_1$}
		& \mathbf{2} & 1 & 1 & 1 & 1 & 2 & 2 & 2 & 2 & 3 & 3 & 3 & 4 & 4 & 4 & 5 & 5 & 5 \\
		\cline{2-19}
		& \mathbf{3} & - & - & - & - & 1 & 1 & 1 & 2 & 2 & 2 & 3 & 3 & 3 & 4 & 4 & 4 & 5 \\
		\cline{2-19}
		& \mathbf{4} & - & - & - & - & - & - & - & 1 & 1 & 1 & 2 & 2 & 2 & 3 & 3 & 3 & 4 \\
		\cline{2-19}
		& \mathbf{5} & - & - & - & - & - & - & - & - & - & - & 1 & 1 & 1 & 2 & 2 & 2 & 3 \\
		\cline{2-19}
		& \mathbf{6} & - & - & - & - & - & - & - & - & - & - & - & - & - & 1 & 1 & 1 & 2 \\  
		\cline{2-19}
		& \mathbf{7} & - & - & - & - & - & - & - & - & - & - & - & - & - & - & - & - & 1 \\   
		\cline{1-19}
		\end{array}}
	\]
	\caption{\label{tab:2} Maximum number of corners of 3-spread strongly stable ideals}
\end{table}
Given $\ell_1$ (the initial degree), the output of each table shows that the maximum number of admissible corners remains eventually unchanged for $t$ consecutive values of $n$ and then increases by 1. For this reason, we have estimated that it may be convenient to decompose $n$ with respect to $t$ by writing $n=d+kt$ for suitable positive integers $d$ and $k$.\\

For later use, recall that the floor function of a real number $x$ is defined as follows:
$$
\lfloor x\rfloor=\max\{n\in\Z:n\le x\}.
$$
In particular, if $-1\le x<0$, $\lfloor x\rfloor=-1$ and if $0\le x<1$, $\lfloor x\rfloor=0$.

\subsection{Methodology and preliminary results}\label{Disc}

Let $I\in \sS_{t,n,\bm{1}}$ such that $\indeg I =2$. We will verify that if one wants $I$ to have the maximal number of extremal Betti numbers, then 
\begin{enumerate}
	\item[-] $G(I)_2=B_t(x_1x_n)$, and 
	\item[-] $n=d+kt$ with $k\ge 3$ and $1\le d\le t$.
\end{enumerate}

\par\noindent
\textbf{Claim.} Set $\omega_0=x_1x_n$ and let $G(I)_2=B_t(x_1x_n)$.
There exist $t$--spread monomials $\omega_1,\dots,\omega_{k+\left\lfloor\frac{d-3}{t}\right\rfloor-1}$ such that 
$$
\omega_j:=\max\Big\{u\in M_{n,j+2,t}:u\notin \bigcup_{i=0}^{j-1}\shad_t^{j-i}(B_t(\omega_{i}))\ \text{and}\ \max(u)=n\Big\},
$$
for $j=1,\dots,k+\left\lfloor\frac{d-3}{t}\right\rfloor-1$.

One can observe that in order to determine $\omega_j$, it is sufficient to find the minimum $v$ of $\shad_t(B_t(\omega_{j-1}))$. Then $\omega_j$ will be the largest $t$--spread monomial of degree $j+2$ with $\max(\omega_j)=n$  following $v$ in the squarefree lexicographic order (see also \cite{AC6}).

\vspace{0,2cm}

In order to prove the {\bf Claim}, we need the next crucial lemma.
\begin{lemma}
	\label{maxbettilemma}
	Let $n, t$ be two positive integers. Let $u=x_{i_1}x_{i_2}\cdots x_{i_d}$, $1\le i_1<i_2<\dots<i_d\le n$ be a $t$--spread monomial of $S=K[x_1,\dots,x_n]$ such that $\max(u)=n$. If $i_{j+1}-i_j=t$, for all $j=1,\dots, d-1$ then $B_t(u)$ is the $t$--spread Veronese ideal of degree $d$.
	%$u$ is the smallest $t$--spread monomial of degree $d$ with $\max(u)=n$ with respect to $<_{\slex}$. 
	Otherwise, if $p=\max\{j:i_{j+1}-i_j>t\}$, then the largest $t$--spread monomial $v$ of degree $d$ of $S$ with $\max(v)=n$ following $u$, with respect to $>_{\slex}$, is
	%  in the squarefree lexicographic order is
	$$
	v=x_{i_1}\cdots x_{i_{p-1}}x_{i_p+1}x_{i_p+1+t}\cdots x_{i_p+1+t(d-p-1)}x_{n}.
	$$
\end{lemma}
\begin{proof}
	If $i_{j+1}-i_j=t$, for all $j=1,\dots,d-1$, then $u=x_{n-t(d-1)}\cdots x_{n-t}x_{n}$. Hence, $u$ is the smallest $t$--spread monomial of degree $d$ with $\max(u)=n$ and $B_t(u)$ is generated by all $t$--spread monomials of degree $d$ of $S$, \emph{{\it i.e.}}, $B_t(u)$ is the $t$--spread Veronese ideal of degree $d$ \cite{EHQ}.
	
	Now, suppose $i_{j+1}-i_j>t$ for some $j$ and let $p=\max\{j:i_{j+1}-i_j>t\}$.
	If $w=x_{s_1}x_{s_2}\cdots x_{s_d}$ is a $t$--spread monomial of degree $d$ with $\max(w)=n$ and $u>_{\slex}w$, then $i_1=s_1,\dots,i_{j-1}=s_{j-1}$ and $i_j<s_j$ for some index $j$. 
	
	It is $j\le p$. Indeed, if $j>p$, then $i_{p+2}-i_{p+1}=\dots=i_d-i_{d-1}=t$. Hence, $i_d-i_j=t(d-j)$, $i_d=s_d=n$, $s_{d}-s_j\ge t(d-j)$, and $i_j<s_j$. Thus,
	$$
	t(d-j)\le s_{d}-s_j=i_d-s_j<i_d-i_j=t(d-j),
	$$
	and so $t(d-j)<t(d-j)$. This is absurd. Hence, $j\le p$.
	
	Therefore, setting
	\begin{equation}\label{eq1:lem}
	%$$
	v=x_{i_1}\cdots x_{i_{p-1}}x_{i_p+1}x_{i_p+1+t}\cdots x_{i_p+1+t(d-p-1)}x_{n},
	%$$
	\end{equation}
	one has
	\[u>_{\slex}v.\]
	
	Moreover, it is easy to verify that $v$ is the monomial we are looking for.
\end{proof}
\vspace{0,3cm}

Lemma \ref{maxbettilemma} will play a key role in getting the $\omega_i$'s in the claim.

In order to simplify the notation, we set $$\displaystyle\Omega_j:=\Big\{u\in M_{n,j+2,t}:u\notin \bigcup_{i=0}^{j-1}\shad_t^{j-i}(B_t(\omega_{i}))\ \text{and}\ \max(u)=n\Big\}$$
and so $\omega_j:=\max\Omega_j$, for all $j$.\\

Next remark will be pivotal for the rest.
\begin{oss}\label{rem1}
	\rm Let $t\ge 2$, $n=d+kt$, $1\le d\le t$ and $k\in\{0,1,2\}$.
	\begin{enumerate}[label=\textup{(\roman*)}]
		\item If $k=0$, then there is no $t$--spread monomial of degree two. Indeed, in such a case $n=d < t+1$.
		\item If $k=1$, then $n=d+t$. Hence, $\omega_0=x_{1}x_{d+t}$ and $\shad_t(B_t(\omega_0))=\emptyset$.
		%we can't construct other monomials.
		\item Let $k=2$, then $\omega_0=x_1x_n=x_1x_{d+2t}$. If $d=1$, then $\shad_t(B_t(\omega_0))=\{x_1x_{1+t}x_{1+2t}\}=M_{n,3,t}$, so $\Omega_1=\emptyset$. If $2\le d\le t$, then $\shad_t(B_t(\omega_0))=B_t(x_1x_{d+t}x_{d+2t})$. Hence, $\min\shad_t(B_t(\omega_0))=x_1x_{d+t}x_{d+2t}$, and by Lemma \ref{maxbettilemma}, $\omega_1 = \max \Omega_1 =x_2x_{2+t}x_{d+2t}$ is the largest $t$--spread monomial $u$ with $\max(u)=n$ following $x_1x_{d+t}x_{d+2t}$, with respect to $>_{\slex}$.
	\end{enumerate}
\end{oss}
\vspace{0,2cm}
\par\noindent
Remark \ref{rem1} points out that the decomposition 
\[n=d+kt,\,\, 1\le d\le t\] 
does not \emph{work well} in the sense of the {\bf Claim}, whenever $k\in\{0,1,2\}$. 

\subsubsection{Basic monomials of the first type}\label{sub1}

In this Subsection, if $S=K[x_1, \ldots, x_n]$, $n=d+kt$ ($k\ge 3$, $1\le d\le t$), setting, $\omega_0 = x_1x_n$, we construct a set of monomials $\omega_1,\dots,\omega_q$ of $S$, $q\le {k+\left\lfloor\frac{d-3}{t}\right\rfloor-1}$, of the type described in the {\bf Claim}. Such  monomials will be called \emph{basic monomials of the first type}.

Let $k\ge 3$. For the sake of clarity, we distinguish two cases.

{\bf Case 1.} Let $k=3$, $n=d+3t$. 

Set $\omega_0=x_1x_n$.
The minimum of $\shad_t(B_t(\omega_0))$ is $u=x_1x_{n-t}x_n=x_1x_{d+2t}x_{d+3t}$. By Lemma \ref{maxbettilemma}, the largest $t$--spread monomial $v$ of degree 3 with $\max(v)=n$ that follows $u$ in the squarefree lexicographic order is $v=x_2x_{2+t}x_{d+3t}$. It is clear that $v= \omega_1= \max \Omega_1$.

Let us discuss the ``distance'' between the last two variables of $\omega_1 = x_2x_{2+t}x_{d+3t}$. We need two consider some cases.

If $d=1$, then $1+3t-(2+t)=2t-1$, $\shad_t(B_t(\omega_1))=\{x_1x_{1+t}x_{1+2t}x_{1+3t}\}$ $=M_{1+3t,4,t}$ and $\omega_2$ does not exist. Hence, we have $k-1=2$ monomials, \emph{{\it i.e.}}, $\omega_0, \omega_1$.

If $d=2$, then $2+3t-(2+t)=2t$, $\shad_t(B_t(\omega_1))=B_t(x_2x_{2+t}x_{2+2t}x_{2+3t})=M_{n,4,t}$ and we cannot construct $\omega_2$. Hence, also in such a case, we have $k-1=2$ monomials.

Finally,  let $3\le d\le t$. In such a case $\shad_t(B_t(\omega_1))=B_t(x_2x_{2+t}x_{d+2t}x_{d+3t})$. Since $\min\shad_t(B_t(\omega_1))= x_2x_{2+t}x_{d+2t}x_{d+3t}$, then, by  Lemma \ref{maxbettilemma}, one has that $\omega_2=x_2x_{3+t}x_{3+2t}x_{d+3t}$. On the other hand, since $d+4-t\le 4$, we have
\begin{eqnarray*}
	|M_{n,5,t}|&=&|M_{d+3t,5,t}|
	%&=&
	=\binom{d+3t-4t+4}{5}=\binom{d+4-t}{5}=0.
\end{eqnarray*}

Thus, $M_{n,5,t}=\emptyset$ and $\omega_3$ does not exist. Hence, in such a case, we have constructed $k=3$ monomials.

\par\noindent
{\bf Case 2.} Let $k\ge 4$. Firstly, we consider an example.

\begin{ese}
	\rm
	Let $n=9$ and $t=2$, we can write $n=d+kt$, with $d=1,k=4$. Then, $\omega_{0}=x_1x_9$ and $\omega_1=x_2x_4x_9$. Observe that in such a case 
	$$
	\Omega_2:=\Big\{u\in M_{9,4,2}:u\notin \shad_2^2(B_2(\omega_{0}))\cup\shad_2(B_2(\omega_{1}))\ \text{and}\ \max(u)=9\Big\}\neq \emptyset.
	$$
	%is not empty. 
	Indeed, $|M_{9,4,2}|=\binom{6}{4}=15$ and the monomials $u\in M_{9,4,2}$ with $\max(u)=9$ are the following ones
	\begin{center}
		$x_1x_3x_5x_9,\ x_1x_3x_6x_9,\ x_1x_3x_7x_9,\ x_1x_4x_6x_9,\ x_1x_4x_7x_9,\ x_1x_5x_7x_9,$\\
		$x_2x_4x_6x_9,\ x_2x_4x_7x_9,\ x_2x_5x_7x_9.$
	\end{center}
	Hence, $\Omega_2=\{x_2x_5x_7x_9\}$ and $\omega_2=x_2x_5x_7x_9$. 
	
	Note that $$\omega_2=x_{2}x_{2+t+1}x_{2+2t+1}x_n = x_{2}x_{3+t}x_{3+2t}x_n,$$
	as in the case $k=3$ for $3\le d\le t$.
\end{ese}\medskip

Assume $k\ge 4$. For $j\ge 1$, let us define the following monomials of $S$ of degree $j+2$
\begin{equation}
\label{eq:omega}
\begin{aligned}
\omega_{j}&:=\bigg(\prod_{i=0}^{j-1}x_{2+i+it}\bigg)x_{(j+1)+jt}x_{d+kt}\\
&=x_2x_{3+t}x_{4+2t}\cdots x_{(j+1)+(j-1)t}x_{(j+1)+jt}x_{d+kt}.
\end{aligned}
\end{equation}

For $j=1,2$, one has:
\begin{align*}
\omega_{1}&=x_2x_{2+t}x_{d+kt},\\\omega_2&=x_2x_{3+t}x_{3+2t}x_{d+kt}.
\end{align*}
It is clear that $\omega_1=\max\Omega_1$ and $\omega_2=\max\Omega_2$. 

The monomials $\omega_j$ are $t$--spread if $j$ satisfies the inequality $(j+1)+jt\le n-t$. 

Let us determine the greatest such an integer, \emph{{\it i.e.}},
$$
j_{\max}=\max\big\{j:(j+1)+jt\le n-t\big\}.
$$
For every $j\in\big\{j:(j+1)+jt\le n-t\big\}$ one has $j(1+t)\le n-t-1$. Therefore,
$$
j\le \frac{n-(1+t)}{1+t}=\frac{n}{1+t}-1,
$$
and
$$
j_{\max}=\left\lfloor\frac{n}{1+t}\right\rfloor-1.
$$

Now, we want to verify that $\omega_j= \max\Omega_i$, for $j =1, \ldots,  j_{\max}$.

Assume $\max\Omega_{j-1}=\omega_{j-1}$. Since,
\[\min\shad_t\big(B_t(\omega_{j-1})\big)=\omega_{j-1}x_{n-t}\
=\bigg(\prod_{i=0}^{j-2}x_{2+i+it}\bigg)x_{j+(j-1)t}x_{d+(k-1)t}x_{d+kt},\]
by Lemma \ref{maxbettilemma}, one has 
$$
\max\Omega_j=\bigg(\prod_{i=0}^{j-2}x_{2+i+it}\bigg)x_{(j+1)+(j-1)t}x_{(j+1)+jt}x_n = \omega_j.
$$

Note that (\ref{eq:omega}) describes also the $\omega_j$'s of the case $k=3$.\\

In the sequel, the monomials $\omega_j$ ($j=1,\dots,j_{\max}$) will be called \textit{basic monomials of the first type}, or also \textit{basic forward monomials}, because each monomial $\omega_j$ is obtained by changing the penultimate variable of the preceding monomial  ($\omega_{j-1}$) of the list, as next example illustrates.
%\end{disc}
\begin{ese}
	\label{es1mainteor}
	\rm Let $n=46$ and $t=3$, we can write $n=1+15t$. We determine $j_{\max}$.
	\begin{align*}
	j_{\max}&=\left\lfloor\frac{n}{1+t}\right\rfloor-1=\left\lfloor\frac{46}{4}\right\rfloor-1=10.
	\end{align*}
	Firstly, we set $\omega_0=x_1x_n=x_1x_{46}$. Then, we have the following $j_{\max}$ ``further monomials'': %``$j_{\max}=10$''given by
	\begin{align*}
	\omega_{j}:=\bigg(\prod_{i=0}^{j-1}x_{2+i+it}\bigg)x_{(j+1)+jt}x_{d+kt}=x_2x_{3+t}\cdots x_{(j+1)+(j-1)t}x_{(j+1)+jt}x_{d+kt},
	\end{align*}
	for all $j=1,\dots,10$. More in details,
	\vspace{0.2cm}
	\small
	\boxalign{\begin{align*}
		\omega_1&=\underline{x_2}\bm{x_5}x_{46}, &&&\omega_6&=x_2x_6x_{10}x_{14}x_{18}\underline{x_{22}}\bm{x_{25}}x_{46},\\
		\omega_2&=x_2\underline{x_6}\bm{x_9}x_{46},&&&\omega_7&=x_2x_6x_{10}x_{14}x_{18}x_{22}\underline{x_{26}}\bm{x_{29}}x_{46},\\
		\omega_3&=x_2x_6\underline{x_{10}}\bm{x_{13}}x_{46},&&&\omega_8&=x_2x_6x_{10}x_{14}x_{18}x_{22}x_{26}\underline{x_{30}}\bm{x_{33}}x_{46},\\
		\omega_4&=x_2x_6x_{10}\underline{x_{14}}\bm{x_{17}}x_{46},&&&\omega_9&=x_2x_6x_{10}x_{14}x_{18}x_{22}x_{26}x_{30}\underline{x_{34}}\bm{x_{37}}x_{46},\\
		\omega_5&=x_2x_6x_{10}x_{14}\underline{x_{18}}\bm{x_{21}}x_{46},&&&\omega_{10}&=x_2x_6x_{10}x_{14}x_{18}x_{22}x_{26}x_{30}x_{34}\underline{x_{38}}\bm{x_{41}}x_{46}.
		\end{align*}}
	\normalsize
	Note that every monomial $\omega_i$ of the list can be obtained by changing the second to last variable of the previous monomial $\omega_{i-1} =x_{q_1}\cdots x_{q_r}$ of the list by adding $1$ to the index of such a variable and inserting a new variable indexed by $q_{r-1}+1+t$.
	
	For example, $\omega_2=x_2x_6\bm{x_9}x_{46}$ and $\omega_3=x_2x_6\underline{x_{10}}x_{13}x_{46}$.\\
	
	Observe, that in this case we can construct another monomial of the kind described in the {\bf Claim}.
	Indeed $\Omega_{11}$ is not empty and it is easy to verify that
	$$
	\omega_{j_{\max}+1}=\omega_{11}=\max\Omega_{11}=x_2x_6x_{10}x_{14}x_{18}x_{22}x_{26}x_{31}x_{34}x_{37}x_{40}x_{43}x_{46}.
	$$
\end{ese}

\subsubsection{Basic monomials of the second type}\label{sub2}
Example \ref{es1mainteor} suggests us the construction of further $t$--spread monomials of $S$ which will be fundamental for our aim. Such monomials will be called \emph{basic monomials of the second type}.

Let us consider the $t$--spread monomial of degree $j_{\max}+2$
\[\omega_{j_{\max}}=x_{i_1}x_{i_2}\cdots x_{i_{(j_{\max}+2)}}\\
=\bigg(\prod_{i=0}^{j_{\max}-1}x_{2+i+it}\bigg)x_{(j_{\max}+1)+j_{\max}t}x_{d+kt}.\]
We observe that $i_{m+1}-i_m=t+1$ for all $m=1,\dots,j_{\max}-1$, and $i_{j_{\max+1}}-i_{j_{\max}}=t$. Moreover,
$$d+kt-[(j_{\max}+1)+j_{\max}t]=n-[(j_{\max}+1)+j_{\max}t]\le 2t.$$ Indeed, if $n-[(j_{\max}+1)+j_{\max}t]>2t$, then $n-t>(j_{\max}+1)+j_{\max}t+t$. Hence, $n-t\ge(j_{\max}+2)+(j_{\max}+1)t$ and $j_{\max}+1$ would be an integer greater than $j_{\max}$ which belongs to the set $\big\{j:(j+1)+jt\le n-t\big\}$. It is an absurd. Finally, $n-[(j_{\max}+1)+j_{\max}t]\le 2t$.

Now, let us examine the integer
\begin{align*}
s=2t-\Big[n-[(j_{\max}+1)+j_{\max}t]\Big]=2t-n+j_{\max}(1+t)+1.
\end{align*}
We need to distinguish two cases: $j_{\max}-1-s\ge 1$, $j_{\max}-1-s<1$.\\

Let $j_{\max}-1-s\ge 1$. In such a case,
\[	v=\min\shad_t(B_t(\omega_{j_{\max}}))=x_{i_1}x_{i_2}\cdots x_{i_{(j_{\max}-1-s)}}\bigg(\prod_{i=k-3-s}^kx_{d+it}\bigg).
\]
By Lemma \ref{maxbettilemma}, $\omega_{j_{\max}+1}$ does exist and it is the largest $t$--spread monomial $u$ of degree $j_{\max}+3$, with $\max(u)=n$ following $v$ in the squarefree lexicographic order:
\[\omega_{j_{\max}+1}= x_{i_{(j_{\max}-1-s)+1}}(v/x_{i_{(j_{\max}-1-s)}}).\]

Now we focus on the variable we are going to ``change'' in $v$ in order to obtain the monomial $\omega_{j_{\max}+1}$:
\begin{align*}
i_{(j_{\max}-1-s)}+1&=2+j_{\max}-2-s+(j_{\max}-2-s)t+1
\\&=2+j_{\max}-2-2t+n-j_{\max}(1+t)-1+(j_{\max}-2-s)t+1
\\&=d+(k-4-s)t.
\end{align*}
Therefore, from Lemma \ref{maxbettilemma} (see (\ref{eq1:lem})), the greatest monomial in $\Omega_{j_{\max}+1}=\Omega_{\left\lfloor\frac{n}{1+t}\right\rfloor}$ is 
$$
\omega_{j_{\max}+1}=\max\Omega_{j_{\max}+1} =  x_{i_1}x_{i_2}\cdots x_{i_{(j_{\max}-2-s)}}\bigg(\prod_{i=k-4-s}^kx_{d+it}\bigg).
$$
We call such monomial the \emph{critic monomial}, since from now on the next monomials we are going to construct are no longer obtained by changing the penultimate variable. 
\begin{ese}
	\rm Let us consider again Example \ref{es1mainteor}.
	In such a case, it is
	$$
	s=2t-\Big[n-[(j_{\max}+1)+j_{\max}t]\Big]=6-[46-41]=1.
	$$
	Since $j_{\max}-1-s\ge1$, then $\omega_{j_{\max}+1}$ exists. Setting $\omega_{j_{\max}}=\omega_{10}=x_{i_1}x_{i_2}\cdots x_{i_{12}}$, then $i_{(j_{\max}-1-s)}+1=i_{10-1-1}+1=i_8+1=31$ and the critic monomial $\omega_{j_{\max}+1}=\omega_{11}$ is the following one
	\begin{align*}
	\omega_{j_{\max}+1}&=x_{i_1}x_{i_2}\cdots x_{i_{(j_{\max}-2-s)}}\bigg(\prod_{i=k-4-s}^kx_{d+it}\bigg)\\&=x_2x_6x_{10}x_{14}x_{18}x_{22}x_{26}\boxed{x_{31}}\underline{x_{34}x_{37}x_{40}}\bm{x_{43}}x_{46}.
	\end{align*}
	Observe that $\omega_{j_{\max}+2}$ exists and $$\omega_{j_{\max}+2}=x_2x_6x_{10}x_{14}x_{19}x_{22}x_{25}x_{28}x_{31}x_{34}x_{37}x_{40}x_{43}x_{46}.$$
\end{ese}

Now, our question is: \emph{How may admissible $t$--spread monomials can we construct starting from $\omega_{j_{\max}+1}$?}

Consider the critic monomial
$$
\omega_{j_{\max}+1}=x_{i_1}x_{i_2}\cdots x_{i_{(j_{\max}-2-s)}}\bigg(\prod_{i=k-4-s}^kx_{d+it}\bigg).
$$
Since $i_{m+1}-i_m=t+1$ for all $m=1,\dots,j_{\max}-3-s$, from Lemma \ref{maxbettilemma} (see (\ref{eq1:lem})), we have
$$
\max\Omega_{j_{\max}+2} = \omega_{j_{\max}+2}=x_{i_1}x_{i_2}\cdots x_{i_{(j_{\max}-2-s-t-1)}}\bigg(\prod_{i=k-4-s-t-1}^kx_{d+it}\bigg).
$$

Proceeding in such a way, we can get the further $\omega_{j_{\max}+1+\nu}$ monomials,
$$
\omega_{j_{\max}+1+\nu}=x_{i_1}x_{i_2}\cdots x_{i_{(j_{\max}-2-s-\nu t)}}\bigg(\prod_{i=k-4-s-\nu (1+t)}^kx_{d+it}\bigg),
$$
as long as $i_{(j_{\max}-1-s-\nu t)}\ge i_1$, {\it i.e.} $j_{\max}-1-s-\nu t\ge 1$.

Now, let us determine
$$
\nu_{\max}=\max\big\{\nu:j_{\max}-1-s-\nu t\ge 1\big\}.
$$
If $\nu$ is such that $j_{\max}-1-s-\nu t\ge 1$, then $\nu t\le j_{\max}-2-s$. Hence
\begin{align*}
\nu t&\le j_{\max}-2-s=j_{\max}-2-2t+n-j_{\max}(1+t)-1
\\&=j_{\max}-2-2t+d+kt-j_{\max}-j_{\max}t-1
\\&=d-3+(k-2-j_{\max})t.
\end{align*}
Thus, we have
$$
\nu_{\max}=\left\lfloor\cfrac{d-3+\left(k-2-j_{\max}\right)t}{t}\right\rfloor=\left\lfloor\frac{d-3}{t}\right\rfloor+k-2-j_{\max}.
$$
So we can construct other $\nu_{\max}$ $t$--spread monomials $\omega_{j_{\max}+2},\dots,\omega_{j_{\max}+\nu_{\max}+1}$. 

%It is worthy of being pointed out that the critic monomial $\omega_{j_{\max}+1}$ is a basic monomial of the second type.

Finally, we have constructed the following $t$--spread monomials of $S$:
\begin{enumerate}
	\item[-] $\omega_0=x_1x_n$ (one monomial);
	\item[-]  $\omega_1,\dots,\omega_{j_{\max}}$ ($j_{\max}$ basic monomials of the first type);
	\item[-]  $\omega_{j_{\max}+1},\omega_{j_{\max}+2},\dots,\omega_{j_{\max}+1+\nu_{\max}} $ ($\nu_{\max}+1$ monomials)
\end{enumerate}
which satisfy the {\bf Claim}. Their total number is
$$
1+j_{\max}+1+\nu_{\max}=1+j_{\max}+1+\left\lfloor\frac{d-3}{t}\right\rfloor+k-2-j_{\max}=k+\left\lfloor\frac{d-3}{t}\right\rfloor.
$$

The monomials $\omega_{j_{\max}+1},\omega_{j_{\max}+2},\dots,\omega_{j_{\max}+1+\nu_{\max}}$ will be called  \textit{basic monomials of second type}, or also \textit{basic backward monomials}, because each of these monomials is obtained, for all $\nu=0,\dots,\nu_{\max}$, by changing the $(j_{\max}-2-s-\nu t+1)-\textup{th}$ variable of the preceding monomial of the list.

Recall that we are considering $t\ge2$. We observe that
$$
k+\left\lfloor\frac{d-3}{t}\right\rfloor=\begin{cases}
k-1&\text{if}\ d=1,2,\\
\phantom{-}k&\text{if}\ 3\le d\le t.
\end{cases}
$$

Let us show that $\omega_{j_{\max}+\nu_{\max}+1}$ is the last monomial which satisfies the {\bf Claim}, {\it i.e.} $\Omega_{j_{\max}+\nu_{\max}+2}=\emptyset$. 

We need to examine some cases.

If $d=1$ or $d=2$, then if one may construct another monomial of the type described in the {\bf Claim}, its degree would be $j_{\max}+\nu_{\max}+4=k+\lfloor\frac{d-3}{t}\rfloor+2=k+1$ and
$$
|M_{n,k+1,t}|=\binom{n-(k+1-1)(t-1)}{k+1}=\binom{d+kt-kt+k}{k+1}=\binom{d+k}{k+1}.
$$

Hence, if $d=1$, one has $|M_{n,k+1,t}|=1$ and $M_{n,k+1,t}=\{x_{1}x_{1+t}\cdots x_{1+kt}\}$. Since, $x_{1}x_{1+t}\cdots x_{1+kt}$ $\in\shad_t^{k-1}(B_t(\omega_0))=\shad_t^{k-1}(B_t(x_1x_{1+kt}))$, then we have $\Omega_{j_{\max}+\nu_{\max}+2}=\emptyset$; whereas, if $d=2$, then $|M_{n,k+1,t}|=k+1$. Moreover, in such a case, $\max M_{n,k+1,t}= x_1x_{1+t}\cdots x_{2+kt}$ and $\min M_{n,k+1,t}=x_2x_{2+t}\cdots x_{2+kt}$.
Let $z\in M_{n,k+1,t}$ with $\max(z)=2+kt$. If $\min(z)=1$, then $z\in\shad_t^{k-1}\big(B_t(\omega_{0})\big)$; if $\min(z)=2$, then $z\in\shad_t^{k-2}\big(B_t(\omega_{1})\big)$. Therefore, $\Omega_{j_{\max}+\nu_{\max}+2}=\emptyset$.

Now, let $3\le d\le t$. If one could construct another monomial of the type described in the {\bf Claim}, then its degree would be equal to $j_{\max}+\nu_{\max}+4=k+\lfloor\frac{d-3}{t}\rfloor+2=k+2$ and
$$
|M_{n,k+2,t}|=\binom{n-(k+2-1)(t-1)}{k+2}=
\binom{k+1+d-t}{k+2}=0.
$$
In fact $k+1+d-t<k+2$. Hence $M_{n,j_{\max}+\nu_{\max}+4,t}=M_{n,k+2,t}=\emptyset$ and $\Omega_{j_{\max}+\nu_{\max}+2}=\emptyset$.

Hence, in every admissible case, $\Omega_{j_{\max}+\nu_{\max}+2}=\emptyset$ and consequently $\omega_{j_{\max}+\nu_{\max}+1}$ is the last $t$--spread monomial of the type described in the {\bf Claim} that one may construct. 

\begin{ese} \label{es:imp}
	\rm 
	We consider again Example \ref{es1mainteor}. In such case, $\nu_{\max}=\left\lfloor\frac{d-3}{t}\right\rfloor+k-2-j_{\max}=-1+15-2-10=2.$ There are $\nu_{\max}=2$
	% remaining 
	monomials of the second type to determine. We set
	\begin{align*}
	\omega_{j_{\max}+1}&=x_{i_1}x_{i_2}\cdots x_{i_{(j_{\max}-2-s)}}\bigg(\prod_{i=k-4-s}^kx_{d+it}\bigg)\\&=x_2x_6x_{10}x_{14}x_{18}x_{22}x_{26}\boxed{x_{31}}\underline{x_{34}}\underline{x_{37}}\underline{x_{40}}\bm{x_{43}}x_{46}.
	\end{align*}
	We determine $\omega_{j_{\max}+2}$ by shifting backward by $t=3$, {\it i.e.}%\newpage
	%$$\curvearrowleft \ \curvearrowleft\ \curvearrowleft\phantom{...}$$
	%\vspace{-1.55cm}
	\begin{align*}
	\omega_{11}&=x_2x_6x_{10}x_{14}\stackon[1pt]{\boxed{x_{18}}x_{22}x_{26}\boxed{x_{31}}}{\Large \curvearrowleft \ \curvearrowleft\ \curvearrowleft} x_{34}x_{37}x_{40}x_{43}x_{46},\\
	&\phantom{x_2x_6x_{10}x_{14}\boxed{x_18}}\big\downarrow\\
	\omega_{12}&=x_2x_6x_{10}x_{14}\boxed{x_{19}}\underline{x_{22}x_{25}x_{28}}\bm{x_{31}}x_{34}x_{37}x_{40}x_{43}x_{46}.
	\end{align*}
	It remains to determine $\omega_{j_{\max}+\nu_{\max}+1}=\omega_{j_{\max}+3}=\omega_{13}$. Shifting backward by $t=3$ again, we have
	
	%$$\curvearrowleft \ \curvearrowleft \ \curvearrowleft\phantom{..................................}$$\vspace{-1.35cm}
	\begin{align*}
	\omega_{12}&=x_2\stackon[1pt]{\boxed{x_6}x_{10}x_{14}\boxed{x_{19}}}{\Large \curvearrowleft \ \curvearrowleft\ \curvearrowleft} x_{22}x_{25}x_{28}x_{31}x_{34}x_{37}x_{40}x_{43}x_{46},\\
	&\phantom{x_2\boxed{x_{11}}}\big\downarrow\\
	\omega_{13}&=x_2\boxed{x_7}\underline{x_{10}x_{13}x_{16}}\bm{x_{19}}x_{22}x_{25}x_{28}x_{31}x_{34}x_{37}x_{40}x_{43}x_{46}.
	\end{align*}
	Hence, we have obtained all the monomials we need.
\end{ese}\medskip	

It may happen that $\omega_{j_{\max}+1}$ does not exist, as next example shows.
\begin{ese}
	\label{esnotcrit}
	\rm Let $n=32$ and $t=5$, we can write $n=2+6t$. Then $j_{\max}=\left\lfloor\frac{n}{1+t}\right\rfloor-1=\left\lfloor\frac{32}{6}\right\rfloor-1=4$. In particular,
	$$
	\omega_{j_{\max}}=\omega_{4}=x_2x_8x_{14}x_{20}x_{25}x_{32}.
	$$
	Observe that $s=2t-\Big[n-[(j_{\max}+1)+j_{\max}t]\Big]=10-[32-25]=3$, and $j_{\max}-1-s=0<1$, therefore $\omega_{j_{\max}+1}$ does not exist. Since
	$$
	\nu_{\max}=\left\lfloor\frac{d-3}{t}\right\rfloor+k-2-j_{\max}=-1,
	$$
	then the total number of monomials constructed is $j_{\max}+1=5$. On the other hand, we can note that  $$j_{\max}+\nu_{\max}+2=k+\left\lfloor\frac{d-3}{t}\right\rfloor=6-1=5.$$ 
\end{ese}\bigskip

It is important to underline that in Example~\ref{esnotcrit}, even though $\omega_{j_{\max}+1}$ does not exist, the formula $k+\left\lfloor\frac{d-3}{t}\right\rfloor$ works well. Such a situation has forced us to analyze the case above.\\

Let $j_{\max}-1-s<1$. 

In such a case, $\omega_{j_{\max}}$ is the last monomial of the type described in the {\bf Claim} that we can construct, and consequently we get $j_{\max}+1=\left\lfloor\frac{n}{1+t}\right\rfloor$ monomials. 

We show that in this case $k+\left\lfloor\frac{d-3}{t}\right\rfloor=j_{\max}+1$. 

In fact, $j_{\max}<2+s=3+2t-n+j_{\max}+j_{\max}t$ and so $j_{\max}t>n-3-2t=(d-3)+(k-2)t$. Hence $j_{\max}>\left\lfloor\frac{d-3}{t}\right\rfloor+k-2$, so $j_{\max}\ge\left\lfloor\frac{d-3}{t}\right\rfloor+k-1$. Moreover,
%For $\nu_{\max}$ we have
$$
\nu_{\max}=\left\lfloor\frac{d-3}{t}\right\rfloor+k-2-j_{\max}\le \left\lfloor\frac{d-3}{t}\right\rfloor+k-2-\left( \left\lfloor\frac{d-3}{t}\right\rfloor+k-1\right)=-1.
$$
If we show that $\nu_{\max}=-1$, then we will have 
$$
k+\left\lfloor\frac{d-3}{t}\right\rfloor=j_{\max}+\nu_{\max}+2=j_{\max}+1=\left\lfloor\frac{n}{1+t}\right\rfloor,
$$
as desired. Indeed, if $\nu_{\max}\le -2$, then $\left\lfloor\frac{d-3}{t}\right\rfloor+k-j_{\max}\le0$, {\it i.e.}
\begin{align}
\label{eq cases}
k&\le j_{\max}-\left\lfloor\frac{d-3}{t}\right\rfloor=\left\lfloor\frac{n}{1+t}\right\rfloor-1-\left\lfloor\frac{d-3}{t}\right\rfloor.
\end{align}
Now, we need to consider two possible cases.

If $d=1$ or $d=2$, then $\left\lfloor\frac{d-3}{t}\right\rfloor=-1$, and $k\le 
\left\lfloor\frac{n}{1+t}\right\rfloor-1-\left\lfloor\frac{d-3}{t}\right\rfloor=\left\lfloor\frac{n}{1+t}\right\rfloor\le\frac{n}{1+t}$. Hence, $k(1+t)\le n=d+kt$ and consequently  $k\le d$; this is absurd since $k\ge 3$ and $d\le 2$.

If $3\le d\le t$, then $\left\lfloor\frac{d-3}{t}\right\rfloor=0$ and $k\le 
\left\lfloor\frac{n}{1+t}\right\rfloor-1-\left\lfloor\frac{d-3}{t}\right\rfloor=\left\lfloor\frac{n}{1+t}\right\rfloor-1\le\frac{n}{1+t}-1$. It follows that  $$k(1+t)\le n-1-t=d+kt-1-t.$$ Hence $k\le d-1-t$. But $d\le t$, so $k\le t-1-t=-1$. This is an absurd. Indeed, $k\ge 3$.

Thus, in each case we have $\nu_{\max}=-1$, as desired.

\section{The main result}\label{sec5}
By the materials in Section \ref{Disc}, we are able to state the main result in the paper.

\begin{teor}
	\label{main Teor}
	Let $n,t, k$ be three positive integers such that $n,t\ge 2$ and $k\ge 3$. Assume
	\[n=d+kt, \quad 1\le d\le t.\]
	Then, every ideal $I\in\sS_{t,n,\bm{1}}$ of initial degree two and with a corner in degree two can have at most
	$$
	k+\left\lfloor\frac{d-3}{t}\right\rfloor=\begin{cases}
	k-1&\text{if}\,\, d=1,2,\\
	k&\text{if}\,\, 3\le d\le t,
	\end{cases}
	$$
	corners.
\end{teor}
\begin{proof}
	Let us consider the $k+\left\lfloor\frac{d-3}{t}\right\rfloor$ monomials of $S=K[x_1,\dots,x_n]$ defined in the {\bf Claim} and consider the $t$--strongly stable ideal %of the polynomial ring $S$
	$$
	I=B_t\big(\omega_0,\omega_1,\dots,\omega_{j_{\max}},\omega_{j_{\max}+1},\omega_{j_{\max}+2},\dots,\omega_{j_{\max}+1+\nu_{\max}}\big).
	$$
	The construction of the monomials $\omega_j$, together with Characterization \ref{betti teor}, guarantees that $I$ is an ideal of $\sS_{t,n,\bm{1}}$ with a corner in degree two and such that
	$$
	|\corn(I)|=k+\left\lfloor\frac{d-3}{t}\right\rfloor=\begin{cases}
	k-1&\text{if}\ d=1,2,\\
	\phantom{-}k&\text{if}\ 3\le d\le t.
	\end{cases}
	$$
	More in details,
	\vspace{0.35cm}
	\begin{center}
		\scalebox{0.8}{$\begin{aligned}
			\corn(I)&=\bigg\{(k_i,\ell_i)\ :\ k_i=n-t(\ell_i-1)-1,\ \ell_i=2+(i-1),\ i=1,\dots,k+\left\lfloor\frac{d-3}{t}\right\rfloor \bigg\}\\
			&=\bigg\{ (n-t-1,2),(n-2t-1,3),\dots,\left(n-\Big(k+\left\lfloor\frac{d-3}{t}\right\rfloor\Big)t-1,k+\left\lfloor\frac{d-3}{t}\right\rfloor+1\right)\bigg\}.
			\end{aligned}$}
	\end{center}
	\vspace{0.35cm}
	It is clear that $|\corn(I)|$ is the maximum number of corners for a $t$--spread strongly stable ideal of $S$. 
\end{proof}\medskip

The results obtained in \cite{AC7} are now consequences of Theorem \ref{main Teor}.

\begin{cor}
	\textup{(\cite[Theorem 2]{AC7})} Let $n\ge 11$ be odd. A $2$--spread strongly stable ideal $I$ of $S$ of initial degree two and with a corner in degree two can have at most $\frac{n-3}{2}$ corners.
\end{cor}
\begin{proof} 
	It is sufficient to write $n=d+kt=1+2k$, with $d=1,t=2$ and $k\ge 5$. 
\end{proof}\medskip
\begin{cor}
	\textup{(\cite[Theorem 4]{AC7}).} Let $n\ge 14$ be even. A $2$--spread strongly stable ideal $I$ of $S$ of initial degree two and with a corner in degree two can have at most $\frac{n-4}{2}$ corners.
\end{cor}
\begin{proof}
	It is sufficient to write $n=d+kt=2+2k$, with $d=2,t=2$ and $k\ge 6$.
\end{proof}\bigskip

From Characterization \ref{betti teor} and Theorem \ref{main Teor}, next result follows.
\begin{teor}
	\label{carcor}
	Let $n=d+kt$ be a positive integer, with $t\ge 2,\ 1\le d\le t$ and $k\ge 3$. Set $\ell_1=2$. Given $r=k+\left\lfloor\frac{d-3}{t}\right\rfloor$ pairs of positive integers
	\begin{equation}
	\label{eq2}
	(k_1,\ell_1),\ (k_2,\ell_2),\ \dots,\ (k_{r},\ell_{r}),
	\end{equation}
	with $1\le k_{r}<k_{r-1}<\dots<k_1\le n-t-1$ and $2=\ell_1<\ell_2<\dots<\ell_{r}\le k+\left\lfloor\frac{d-3}{t}\right\rfloor+1$, then there exists a $t$--spread strongly stable ideal of $S=K[x_1,\dots,x_n]$ of initial degree $\ell_1=2$ and with the pairs in \textup{(\ref{eq2})} as corners if and only if $k_j+t(\ell_j-1)+1=n$, for all $j=1,\dots,k+\left\lfloor\frac{d-3}{t}\right\rfloor$.
\end{teor}

We finish this Section with an example which illustrates our methods.

\begin{ese}
	\rm Let $n=14$ and $t=3$, we can write $n=2+4t$. We determine $j_{\max}$ and $\nu_{\max}$.
	\begin{align*}
	j_{\max}&=\left\lfloor\frac{n}{1+t}\right\rfloor-1=\left\lfloor\frac{14}{4}\right\rfloor-1=2,\\
	\nu_{\max}&=\left\lfloor\frac{d-3}{t}\right\rfloor+k-2-j_{\max}=-1.
	\end{align*}
	Since $\nu_{\max}=-1$, then the critic monomial does not exist. Setting,  $\omega_{0}=x_1x_n=x_1x_{14}$, then, we have two \emph{forward monomials}
	%\small
	%\boxalign{
	\begin{align*}
		\omega_1&=\underline{x_2}\bm{x_5}x_{14},\\
		\omega_2&=x_2\underline{x_6}\bm{x_9}x_{14}.
		\end{align*}%}
	%\normalsize
	Hence %, we have constructed the $3$--spread strongly stable ideal we are looking for
	\begin{align*}
	I=&B_3\big(x_1x_{14},x_2x_4x_{14},x_2x_5x_7x_{14}\big)\\
	=&\big(x_1x_4,x_1x_5,x_1x_6,x_1x_7,x_1x_8,x_1x_9,x_1x_{10},x_1x_{11},x_1x_{12},x_1x_{13},\bm{x_1x_{14}},
	\\&x_2x_5x_8,x_2x_5x_9,x_2x_5x_{10},x_2x_5x_{11},x_2x_5x_{12},x_2x_5x_{13},\bm{x_2x_5x_{14}},
	\\&x_2x_6x_{9}x_{12},x_2x_6x_{9}x_{13},\bm{x_2x_6x_9x_{14}}\big)
	\end{align*}
	is the $3$--spread strongly stable ideal we are looking for. The highlithed monomials are the $3$--spread Borel generators of $I$. The Betti diagram of $I$ is
	
	\begin{center}
		\begin{tabular}{ccccccccccccc}
			&&0&1&2&3&4&5&6&7&8&9&10\\ \hline
			2&:&11&55&165&330&462&462&330&165&55&11&1\\
			3&:&7&28&56&70&56&28&8&1&-&-&-\\
			4&:&3&9&10&5&1&-&-&-&-&-&-\\
		\end{tabular}
	\end{center}	
\end{ese}

\section{The general initial degree case}
\label{sec6}
Theorem \ref{main Teor} gives the maximal number of corners allowed for a $t$--spread strongly stable ideal whenever the initial degree of the ideal is two.

Nevertheless, it is worthy to see how this number changes with respect to the initial degree of the given $t$--spread strongly stable ideal $I$.

In this Section, if $I$ is $t$--spread strongly stable ideal, we focus on $\indeg I= \ell_1\ge 3$ pointing out the differences with the case $\ell_1=2$ (Theorem~\ref{main Teor}).
\begin{teor}
	\label{generalcase}
	Let $n,t, k$ be three positive integers such that $n,t\ge 2$ and $k\ge 3$. Assume
	\[n=d+kt, \quad 1\le d\le t.\]
	Then, every ideal $I\in\sS_{t,n,\bm{1}}$ of initial degree $\ell_1$, $3\le\ell_1\le k+\left\lfloor\frac{d-2}{t}\right\rfloor+1$, and with a corner in degree $\ell_1$ can have at most
	$$
	k+\left\lfloor\frac{d-2}{t}\right\rfloor-(\ell_1-2)
	$$
	corners.
\end{teor}
\begin{proof}
	The proof is very similar to that of Theorem \ref{main Teor}. We prove the existence of a $t$--spread strongly stable ideal $I\in\sS_{t,n,\bm{1}}$ of initial degree $\ell_1$ generated in degrees $\ell_1,\ell_1+1,\dots,k+\left\lfloor\frac{d-2}{t}\right\rfloor-(\ell_1-2)+\ell_1-1=k+\left\lfloor\frac{d-2}{t}\right\rfloor+1$ and such that
	$$
	|\corn(I)|=k+\left\lfloor\frac{d-2}{t}\right\rfloor-(\ell_1-2).
	$$
	
	Firstly, we set $\omega_0=x_1x_{1+t}\cdots x_{1+(\ell_1-2)t}x_{n}$ and $G(I)_2=B_t(\omega_0)$.

	We claim that for all $j=1,2,\dots,k+\left\lfloor\frac{d-2}{t}\right\rfloor+1$, there exist $t$--spread monomials $\omega_1,\omega_2,\dots$, $\omega_{k+\left\lfloor\frac{d-2}{t}\right\rfloor-(\ell_1-2)-1}$ such that 
	$$
	\omega_j=\max_{>_\slex}\Big\{u\in M_{n,j+\ell_1,t}:u\notin \bigcup_{i=0}^{j-1}\shad_t^{j-i}(B_t(\omega_{i}))\ \text{and}\ \max(u)=n\Big\}.
	$$
	
	Set
	$$
	\Omega_j:=\Big\{u\in M_{n,j+\ell_1,t}:u\notin \bigcup_{i=0}^{j-1}\shad_t^{j-i}(B_t(\omega_{i}))\ \text{and}\ \max(u)=n\Big\}.
	$$
	
	Let us consider the case $k=3$.

	If $d=1$, then $n=1+3t$ and $\ell_1=3$. Then, we set $\omega_0=x_1x_{1+t}x_{1+3t}$. Since $M_{1+3t,4,t}=\{x_1x_{1+t}x_{1+2t}x_{1+3t}\}$ and $x_1x_{1+t}x_{1+2t}x_{1+3t}\in\shad_t(B_t(\omega_0))$, we cannot construct $\omega_1$ and, in such a case, a $t$--spread strongly stable ideal with $\ell_1=3$ can have at most one corner. 
	
	If $d\ge2$, then $n=d+3t$.
	
	In particular, if $d=2$, then $3\le\ell_1\le k+\left\lfloor\frac{d-3}{t}\right\rfloor+1=3$. 
	Since $\ell_1=3$, we set $\omega_0=x_1x_{1+t}x_{d+3t}$ and $\omega_1=x_1x_{2+t}x_{2+2t}x_{d+3t}$. Hence, we get two corners.
	
	If $d\ge 3$, then $k+\left\lfloor\frac{d-3}{t}\right\rfloor+1=4$ and $\ell_1\in\{3,4\}$. %If $\ell_1=2$, by Theorem \ref{main Teor} we can have at most $k=3$ corners. 
	If $\ell_1=3$, we set $\omega_0=x_1x_{1+t}x_{d+3t}$, and $\omega_1=x_1x_{2+t}x_{2+2t}x_{d+3t}$. Since $|M_{n,5,t}|=\binom{d+4-t}{5}=0$, we cannot construct $\omega_2$ and we can have at most two corners. If $\ell_1=4$, setting $\omega_0=x_1x_{1+t}x_{1+2t}x_{d+3t}$, since $M_{n,5,t}=\emptyset$, we can have at most one corner.
	
	Let $k\ge 4$. For $j\geq 1$, we consider the monomials
	\begin{equation}
	\label{eq2:omega}
	\begin{aligned}
	\omega_{j}&:=x_1x_{1+t}\cdots x_{1+(\ell_1-3)t}\bigg(\prod_{i=0}^{j-1}x_{2+i+(\ell_1-2+i)t}\bigg)x_{(j+1)+(\ell_1-2+j)t}x_{d+kt}
	\\ &=x_1x_{1+t}\cdots x_{1+(\ell_1-3)t}x_{2+(\ell_1-2)t}\cdots x_{(j+1)+(\ell_1-2+j)t}x_{d+kt}.
	\end{aligned}
	\end{equation}
	
	The monomials $\omega_j$ are $t$--spread as long as $j$ is such that $(j+1)+(\ell_1-2+j)t\le n-t$. We determine the greatest such an integer. We have
	$$
	j_{\max}=\max\big\{j:(j+1)+(\ell_1-2+j)t\le n-t\big\}.
	$$
	Proceeding as in the initial degree two case, we have that
	$$
	j_{\max}=\left\lfloor\frac{n-(\ell_1-2)t}{1+t}\right\rfloor-1,
	$$
	and $\omega_j=\max\Omega_j$, for all $j=1,\dots,j_{\max}$. Now, let
	\begin{align*}
	\omega_{j_{\max}}&=x_{i_1}x_{i_2}\cdots x_{i_{(j_{\max}+\ell_1)}}\\
	&=x_1x_{1+t}\cdots x_{1+(\ell_1-3)t}\bigg(\prod_{i=0}^{j_{\max}-1}x_{2+i+(\ell_1-2+i)t}\bigg)x_{(j_{\max}+1)+(\ell_1-2+j_{\max})t}x_{d+kt}.
	\end{align*}
	We have
	$$
	d+kt-[(j_{\max}+1)+(\ell_1-2+j_{\max})t]=n-[(j_{\max}+1)+(\ell_1-2+j_{\max})t]\le 2t
	$$
	and 
	\begin{align*}
	s&=2t-\Big[n-[(j_{\max}+1)+(\ell_1-2+j_{\max})t]\Big]\\&=2t-n+j_{\max}(1+t)+1+(\ell_1-2)t.
	\end{align*}
	Hence, it follows that
	\[i_{(j_{\max}+\ell_1-3-s)}+1=2+j_{\max}-2-s+(\ell_1-2+j_{\max}-2-s)t+1 = d+(k-4-s)t.\]
	
	Now, we distinguish two cases: $j_{\max}+\ell_1-3-s\ge\ell_1-2$, $j_{\max}+\ell_1-3-s<\ell_1-2$.
	
	Let  $j_{\max}+\ell_1-3-s\ge\ell_1-2$. As in Theorem \ref{main Teor},
	\begin{equation}\label{eq3:omega}
	\max \Omega_{j_{\max}+1} = \omega_{j_{\max}+1}=x_{i_1}x_{i_2}\cdots x_{i_{(j_{\max}+\ell_1-4-s)}}\bigg(\prod_{i=k-4-s}^kx_{d+it}\bigg).
	\end{equation}
	Observe that $i_{\ell_1-2}=1+(\ell_1-3)t$ and $i_{\ell_1-1}=2+(\ell_1-2)t$.
	%Recall that $\ell_1>2$.
	Since
	$$
	i_{m+1}-i_m=\begin{cases}
	t,&\text{for}\ m=1,\dots,\ell_1-3,\\
	t+1,&\text{for}\ m=\ell_1-2,\dots,j_{\max}+\ell_1-5-s,
	\end{cases}
	$$
	then
	$$
	\max \Omega_{j_{\max}+2} = \omega_{j_{\max}+2}=x_{i_1}x_{i_2}\cdots x_{i_{(j_{\max}+\ell_1-4-s-t)}}\bigg(\prod_{i=k-4-s-t-1}^kx_{d+it}\bigg).
	$$
	Finally, we can construct the monomials
	$$
	\omega_{j_{\max}+1+\nu}=x_{i_1}x_{i_2}\cdots x_{i_{(j_{\max}+\ell_1-4-s-\nu t)}}\bigg(\prod_{i=k-4-s-\nu (1+t)}^kx_{d+it}\bigg),
	$$
	as long as $i_{(j_{\max}+\ell_1-3-s-\nu t)}\ge i_{\ell_1-2}$, {\it i.e.} $j_{\max}+\ell_1-3-s-\nu t\ge\ell_1-2$, {\it i.e.} $j_{\max}-1-s-\nu t\ge0$.
	Now, let us determine
	$$
	\nu_{\max}=\max\big\{\nu:j_{\max}-1-s-\nu t\ge0\big\}.
	$$
	If $\nu$ is such that $j_{\max}-1-s-\nu t\ge 0$, then $\nu t\le j_{\max}-1-s$. Hence
	\begin{align*}
	\nu t&\le j_{\max}-1-s=j_{\max}-1-2t+n-j_{\max}(1+t)-1-(\ell_1-2)t
	\\&=j_{\max}-1-2t+d+kt-j_{\max}-j_{\max}t-1-(\ell_1-2)t
	\\&=d-2+\big(k-2-j_{\max}-(\ell_1-2)\big)t.
	\end{align*}
	We have
	\small
	$$
	\nu_{\max}=\left\lfloor\cfrac{d-2+\big(k-2-j_{\max}-(\ell_1-2)\big)t}{t}\right\rfloor=\left\lfloor\frac{d-2}{t}\right\rfloor+k-2-j_{\max}-(\ell_1-2),
	$$
	\normalsize
	and we can construct further \small $\nu_{\max}$ $t$--spread monomials \small $\omega_{j_{\max}+2},\dots,\omega_{j_{\max}+\nu_{\max}+1}$. \normalsize
	
	Finally, we have constructed the 
	$1+j_{\max} +1+\nu_{\max}$ monomials $$\omega_0, \omega_1,\dots,\omega_{j_{\max}}, \omega_{j_{\max}+1}, \omega_{j_{\max}+2},\dots,\omega_{j_{\max}+1+\nu_{\max}}$$ which satisfy our claim.
	Note that
	\begin{align*}
	1+j_{\max}+1+\nu_{\max}&=1+j_{\max}+1+\left\lfloor\frac{d-2}{t}\right\rfloor+k-2-j_{\max}-(\ell_1-2)\\&=k+\left\lfloor\frac{d-2}{t}\right\rfloor-(\ell_1-2).
	\end{align*}
	As in the initial degree two case, $\omega_{j_{\max}+\nu_{\max}+1}$ is the last monomial of the type required in the claim, that we can construct.
	
	Now, suppose $j_{\max}+\ell_1-3-s<\ell_1-2$.\\
	Again, using the same arguments as in in the initial degree two case, one can show that $j_{\max}\ge\left\lfloor\frac{d-2}{t}\right\rfloor+k-1-j_{\max}-(\ell_1-2)$ and $\nu_{\max}\le -1$ and that $\nu_{\max}=-1$. 
	
	Hence, in such a case, we get $j_{\max}+1$ monomials ($\omega_0, \omega_1,\dots,\omega_{j_{\max}}$). Moreover,
	$$
	j_{\max}+1=j_{\max}+1+1+\nu_{\max}=k+\left\lfloor\frac{d-2}{t}\right\rfloor-(\ell_1-2).
	$$
	It is important to underline that in both cases we determine $k+\left\lfloor\frac{d-2}{t}\right\rfloor-(\ell_1-2)$ monomials, also when the critic monomial  $\omega_{j_{\max}+1}$ does not exist.
	
	Setting
	$$
	I=B_t\big(\omega_0,\omega_1,\dots,\omega_{j_{\max}},\omega_{j_{\max}+1},\omega_{j_{\max}+2},\dots,\omega_{j_{\max}+1+\nu_{\max}}\big),
	$$
	the existence of the monomials $\omega_j$, together with Characterization \ref{betti teor}, guarantees that the ideal $I$ is an ideal of $\sS_{t,n,\bm{1}}$ with a corner in degree $\ell_1$ in $S$ and such that
	$$
	|\corn(I)|=k+\left\lfloor\frac{d-2}{t}\right\rfloor-(\ell_1-2).
	$$
\end{proof}\bigskip

The next result (analogous to Theorem~\ref{carcor}) covers the case $t=1$ in \cite{AC6} and the cases $t\geq 2$ in this paper. % of initial degree greater than $2$.

%From Characterization \ref{betti teor} and Theorem \ref{generalcase}, the following generalization of Theorem \ref{carcor} follows.
\begin{teor}
	\label{carcor'}
	Let $n=d+kt$ be a positive integer, with $t\ge 1, 1\le d\le t$ and $k\ge 4$. Let $1\le r\le k+\left\lfloor\frac{d-3}{t}\right\rfloor$ be an integer. Given $r$ pairs of positive integers
	\begin{equation}
	\label{eq2'}
	(k_1,\ell_1),\ (k_2,\ell_2),\ \dots,\ (k_{r},\ell_{r}),
	\end{equation}
	with $1\le k_{r}<k_{r-1}<\dots<k_1\le n-t-1$ and $2\le \ell_1<\ell_2<\dots<\ell_{r}\le k+\left\lfloor\frac{d-3}{t}\right\rfloor+1$, then there exists a $t$--spread strongly stable ideal of $S=K[x_1,\dots,x_n]$ of initial degree $\ell_1$ and with the pairs in \textup{(\ref{eq2'})} as corners if and only if $k_j+t(\ell_j-1)+1=n$, for all $j=1,\dots,r$.
\end{teor}\medskip

\begin{ese}
	\rm Let $n=138$ and $t=11$, we can write $n=6+12t$. Let $\ell_1=5$. We have:
	\begin{align*}
	j_{\max}&=\left\lfloor\frac{n-(\ell_1-2)t}{1+t}\right\rfloor-1=\left\lfloor\frac{105}{12}\right\rfloor-1=7,\\
	\nu_{\max}&=\left\lfloor\frac{d-2}{t}\right\rfloor+k-2-j_{\max}-(\ell_1-2)=0. %0+12-2-7-3=0.
	\end{align*}
	Firstly, we set $\omega_{0}=x_1x_{1+t}\cdots x_{1+(\ell_1-2)t}x_n=x_1x_{12}x_{23}x_{34}x_{138}$. Then we determine the $j_{\max}=7$ monomials given by (\ref{eq2:omega}).
	More precisely,
	\footnotesize
		\boxalign{\begin{alignat*}{3}
			\omega_1&=x_1x_{12}x_{23}\underline{x_{35}}\bm{x_{46}}x_{138},&&& \\
			\omega_2&=x_1x_{12}x_{23}x_{35}\underline{x_{47}}\bm{x_{58}}x_{138},&&& \omega_5&=x_1x_{12}x_{23}x_{35}x_{47}x_{59}x_{71}\underline{x_{83}}\bm{x_{94}}x_{138},\\
			\omega_3&=x_1x_{12}x_{23}x_{35}x_{47}\underline{x_{59}}\bm{x_{70}}x_{138},&&& \omega_6&=x_1x_{12}x_{23}x_{35}x_{47}x_{59}x_{71}x_{83}\underline{x_{95}}\bm{x_{106}}x_{138},\\
			\omega_4&=x_1x_{12}x_{23}x_{35}x_{47}x_{59}\underline{x_{71}}\bm{x_{82}}x_{138},\ \ &&&\omega_7&=x_1x_{12}x_{23}x_{35}x_{47}x_{59}x_{71}x_{83}x_{95}\underline{x_{107}}\bm{x_{118}}x_{138}.
			\end{alignat*}}
	\normalsize
	Following Theorem~\ref{generalcase}, we consider the integer
	$$
	s=2t-\Big[n-[(j_{\max}+1)+(\ell_1-2+j_{\max})t]\Big]=22-[138-118]=2.
	$$
	Since $j_{\max}+\ell_1-3-s\ge\ell_1-2$, then $\omega_{j_{\max}+1}$ exists. Let $\omega_{j_{\max}}=\omega_7=x_{i_1}x_{i_2}\cdots x_{i_{j_{\max}+\ell_1}}=x_{i_1}x_{i_2}\cdots x_{i_{12}}$, then $i_{(j_{\max}+\ell_1-3-s)}+1=i_{7+5-3-2}+1=i_{7}+1=72$ and $\omega_{j_{\max}+1}=\omega_8$ is given by (\ref{eq3:omega}), \emph{i.e.}
	\begin{align*}
	\omega_8&=x_{i_1}x_{i_2}\cdots x_{i_{(j_{\max}+\ell_1-4-s)}}\bigg(\prod_{i=k-4-s}^kx_{d+it}\bigg)\\
	&=x_1x_{12}x_{23}x_{35}x_{47}x_{59}\boxed{x_{72}}\underline{x_{83}x_{94}x_{105}x_{116}}\bm{x_{127}}x_{138}.
	\end{align*}
	This is the last monomial that we can determine, since $\nu_{\max}=0$. Finally, $I=B_{11}(\omega_0,\dots,\omega_8)$ is the desired $t$--spread strongly stable ideal.
\end{ese}

\section{Conclusions and Perspectives}
\label{sec7}
In this paper, following the approach used in \cite{AC6} and \cite{AC7}, we have discussed the extremal Betti numbers of $t$--spread strongly stable ideals and we have determined the maximal number of admissible corners of a $t$--spread strongly stable ideal given its initial degree.
As in \cite{AC7}, it is important to ``decompose'' the integer $n$ with respect to $t$. In a certain way, we have divided $n$ by $t$ forcing the rest of the division to lie in the set $\{1,\dots,t\}$.
In \cite{AC6}, a numerical characterization of the possible extremal Betti numbers (values as well
as positions) of the class of squarefree strongly stable ideals was given. Theorem \ref{carcor'} characterizes the positions of the extremal Betti numbers of the class of $t$--spread strongly stable ideals in the Betti diagram. Nothing is known about the possible values of the extremal Betti numbers of such a class of ideals. This question is currently under investigation by the authors of this paper.

\end{document}